\documentclass{amsart}
\usepackage{amsmath, amssymb,amsthm, amsfonts}
\usepackage[margin=3.6 cm]{geometry}

\def\r{\mathcal R}
\def\R{\mathbb R}

\def\C{\mathbb C} 

\def\F{\mathbb F}

\def\z{{\bf z}}
\def\w{{\bf w}}

\def\V{\mathbb V}

\def\F{\mathbb F}
\def\X{\mathbb X}
\def\A{\mathbb A}

\def \ch{{\bf H}_{\C}}

\def \h{{\bf H}_{\F}}
\def\R{\mathbb R}
\def\V{\mathbb V}

\def\g{\mathcal G}
\newcommand{\SL}{\mathrm{SL}}
\newcommand{\SU}{\mathrm{SU}}

\def \a{\mathcal A}
\def\t{\mathcal T}
\def\P{\mathbb P}

\def \ab {{\bf a}_B}
\def \aa {{\bf a}_A}
\def \ra {{\bf r}_A}
\def \rb {{\bf r}_B}

\def \xb{{\bf x}_B}
\def \yb{{\bf y}_B}

\def \a {{\bf a}}
\def \r  {{\bf r}}
\def \x  {{\bf x}}
\def \y  {{\bf y}}

\newtheorem{theorem}{Theorem}[section]
\newtheorem{lemma}[theorem]{Lemma}
\newtheorem{prop}[theorem]{Proposition}
\theoremstyle{definition}

\theoremstyle{remark}

\numberwithin{equation}{section}
\theoremstyle{plain}

\newtheorem{corollary}[theorem]{Corollary}

\numberwithin{equation}{section}

\newcommand{\secref}[1]{Section~\ref{#1}}
\newcommand{\thmref}[1]{Theorem~\ref{#1}}
\newcommand{\lemref}[1]{Lemma~\ref{#1}}

\newcommand{\propref}[1]{Proposition~\ref{#1}}
\newcommand{\corref}[1]{Corollary~\ref{#1}}
\newcommand{\eqnref}[1]{~{\textrm(\ref{#1})}}

\begin{document}

\title[ Fenchel-Nielsen Coordinates of Surface Group Representations]{ On Fenchel-Nielsen Coordinates of Surface Group Representations into $\SU(3,1)$} 
\author[Krishnendu Gongopadhyay  \and Shiv Parsad]{Krishnendu Gongopadhyay \and
 Shiv Parsad}
\address{Indian Institute of Science Education and Research (IISER) Mohali,
 Knowledge City,  Sector 81, S.A.S. Nagar 140306, Punjab, India}
\email{krishnendug@gmail.com, krishnendu@iisermohali.ac.in}
\address{Indian Institute of Science Education and Research (IISER) Bhopal, 
Bhopal Bypass Road, Bhauri
Bhopal 462 066
Madhya Pradesh, India} 
\email{parsad.shiv@gmail.com}
 \thanks{Gongopadhyay acknowledges partial support from NBHM India, Grant NBHM/R.P.
7/2013/Fresh/992.}
\date{\today}
 \subjclass[2010]{Primary 20H10; Secondary 30F40, 15B57, 51M10}
\keywords{ complex hyperbolic space, surface group representations,  traces}
\begin{abstract}
Let $\Sigma_g$ be a compact, connected, orientable surface of genus $g \geq 2$. We ask for a  parametrization of the discrete, faithful, totally loxodromic representations in the deformation space  ${\rm  Hom}(\pi_1(\Sigma_g), {\rm SU}(3,1))/{\rm SU}(3,1)$. We show that such a representation, under some hypothesis, can be determined by $30g-30$ real parameters. 
\end{abstract}
\maketitle

\section{Introduction}

Let $\Sigma_g$ be a closed, connected, orientable surface of genus $g \geq 2$. Let $\pi_1(\Sigma_g)$ be the  fundamental group of $\Sigma_g$. The classical Teichm\"uller space can be considered as the space of discrete, faithful, totally loxodromic  representations of $\pi_1(\Sigma_g)$ into ${\rm SL}(2,\R)$ up to conjugacy.  To construct the Fenchel-Nielsen coordinates of the classical Teichm\"uller space, one starts by specifying a curve system of $3g-3$ simple closed curves on $\Sigma_g$. The complement of such curve system decomposes the surface into $2g-2$ three-holed spheres. A three-holed sphere is also known as a pair of pants in the literature. The Fenchel-Nielsen coordinates provide the degrees of freedom that are needed to glue these several pairs of pants in order to construct a hyperbolic surface. Given a discrete, faithful, totally loxodromic representation $\rho: \pi_1(\Sigma_g) \to {\rm SL}(2, \R)$, a pair of pants in the above pants-decomposition of the surface corresponds to a two-generator subgroup $\langle A, B \rangle$  generated by 
loxodromic elements $A$ and $B$ such that $AB$ is also loxodromic. The loxodromic elements $A$, $B$, $B^{-1} A^{-1}$ correspond to the boundary components of the pair of pants. Such a group is called a $(0,3)$ group in the literature.  It follows from a classical work of Fricke \cite{f} and Vogt \cite{v} that a $(0,3)$ group in ${\rm SL}(2, \R)$ is completely determined by the traces of the generators and their product. For an up-to-date exposition of this work, see Goldman \cite{gold2}. The gluing of the pairs of pants correspond to gluing of these $(0,3)$ groups. In the gluing process there are traces of these loxodromics, along with rotations of the peripheral or the boundary curves during the gluing. These rotation angles are called twist-bend parameters. The traces of the loxodromics along with the twist-bend parameters determine a representation $\rho: \pi_1(\Sigma_g) \to {\rm SL}(2, \R)$ completely up to conjugation.

\medskip  Let $\ch^n$ be the $n$ dimensional complex hyperbolic space. The group ${\rm SU}(n,1)$ acts as the holomorphic isometry group of $\ch^n$. A discrete, faithful, geometrically finite and totally loxodromic representation of a surface group into ${\rm SU}(n,1)$ is called a \emph{complex hyperbolic quasi-Fuchsian representation}. The geometry of these representations is mostly unknown. In the last two decades, there have been some understanding of these representations when the target group is ${\rm SU}(2,1)$, and a conjectural picture of the representation space has been evolved, see Parker-Platis \cite{pp} and Schwartz \cite{sc} for surveys. However, when the target group is ${\rm SU}(n,1)$, $n \geq 3$, not much is known. 

Parker and Platis \cite{pp} generalized the Fenchel-Nielsen coordinates for discrete, faithful, geometrically finite and totally loxodromic representations of $\pi_1(\Sigma_g)$ into the group ${\rm SU}(2,1)$. As a starting point of their Fenchel-Nielsen coordinate system,  Parker and Platis \cite[Theorem 7.1]{pp} proved a generalization of the result of Fricke and Vogt for 
 Zariski-dense free subgroups of ${\rm SU}(2,1)$ generated by two loxodromic elements. Parker and Platis followed an approach that uses traces of the generators and a point on the cross-ratio variety. In another approach, it follows from the work of Lawton \cite{law}, Khoi \cite{khoi}, Wen \cite{wen}  and Will \cite{will1, will2} that a two-generator free Zariski dense subgroup of ${\rm SU}(2, 1)$ is determined by traces of the generators and the traces of three more compositions of the generators. For a survey of these results, see \cite{parker, will3}. 

There have been generalization of Fenchel-Nielsen coordinates in three dimensional real hyperbolic geometry and projective geometry as well. Tan \cite{tan} and Kourouniotis \cite{ko} constructed Fenchel-Nielsen coordinates for quasi-Fuchsian representations of $\pi_1(\Sigma_g)$ into ${\rm SL}(2, \C)$. Goldman \cite{go3} generalized Fenchel-Nielsen coordinates on the space of convex real projective structures on $\Sigma_g$. Recently, Strubel \cite{stru} has developed Fenchel-Nielsen coordinates for
representations of $\pi_1(\Sigma_g)$ into ${\rm Sp}(2n,\R)$ with maximal Toledo invariant. 

In this work we intend to generalize the work of Parker and Platis \cite{pp} for representations of $\pi_1(\Sigma_g)$ into ${\rm SU}(3,1)$. The starting point, as in the classical case, is to parametrize $(0,3)$ subgroups of ${\rm SU}(3,1)$. However,  for two-generator subgroups in ${\rm SU}(3,1)$,  traces  and cross-ratios of the generators are not sufficient to determine the subgroup up to conjugacy. For the determination of $(0, 3)$ subgroups in ${\rm SU}(3,1)$, one needs to look for more conjugacy invariants that should specify a pair of generators. For this purpose, we use new invariants  which are generalizations of Goldman's eta invariants. As we shall see, for `generic' representations, called \emph{tame representations}, our invariants fit together nicely and they provide Fenchel-Nielsen type coordinates to specify such representations up to conjugacy. 

\medskip 
 Let $\C^{3,1}$ be the vector space $\C^4$ equipped with a non-degenerate Hermitian form $\langle.,. \rangle$ of signature $(3,1)$.  Then $\ch^3$ is the projectivization of the set of  vectors $v$ such that $\langle v, v \rangle<0$.  The boundary $\partial \ch^3$ is the projectivization of the null vectors. Following Goldman \cite{gold} recall that a \emph{$k$-dimensional complex totally geodesic subspace} of $\ch^3$ or a \emph{$\C^k$-plane} is the projectivization of a copy of $\C^{k,1}$ in $\C^{3,1}$, $k=1,2$. A  $\C^1$-plane is simply called a \emph{complex geodesic}.  
 A \emph{$\C^k$-chain} is the boundary of a $\C^k$-plane in $\ch^3$; a $\C^1$-chain is simply called a \emph{chain}. 
A positive vector $c$ (i.e., $\langle c, c \rangle >0$)    is polar to a $\C^2$-plane $C$ if the lift of $C$ in $\C^{3,1}$ is the orthogonal complement of $c$. The positive vector $c$ is polar to a $\C^2$-chain $L$ if $L$ is the boundary of a $\C^2$-plane $C$ that is polar to $c$.

For four distinct points $z_1$, $z_2$, $z_3$ and $z_4$ in $\partial \ch^3$, the \emph{Kor\'anyi-Reimann cross-ratio} is defined by: 
\begin{equation}\label{ecr} \X(z_1, z_2, z_3, z_4)=\frac{\langle {\bf z}_3, {\bf z}_1 \rangle \langle {\bf z}_4, \bf z_2 \rangle} { \langle {\bf z}_4, {\bf z}_1\rangle \langle   {\bf z}_3, {\bf z}_2 \rangle},\end{equation}
where ${\bf z_i}$ is a lift of $z_i$ in $\C^{3,1}$. 
 For more details on cross-ratios, see \cite{gold}. We extend the above definition to define invariants for the ``generic case" that includes three null vectors and one positive vector in $\C^{3,1}$. For a loxodromic element $A$, we denote by ${\bf a}_A$, ${\bf r}_A$ the null eigenvectors of $A$  corresponding to the fixed points and let ${ \bf x}_A$ and ${ \bf y}_A$ correspond to the positive eigenvectors of $A$.

 Let $A$, $B$ be two loxodromic elements in ${\rm SU}(3,1)$. Then corresponding to the fixed points of $A$ and $B$, there are three cross-ratios $\X_k(A, B)$, $k=1,2,3$, that determine the four points uniquely. All such cross-ratios corresponding to pair of loxodromic elements form a variety, called the \emph{cross-ratio variety}. It follows that every point in this variety has five real degrees of freedom, see \propref{crp}.    
The pair $(A, B)$ is called \emph{non-singular}  if 
\begin{itemize} 
\item[(i)] $A$ and $B$ have no common fixed point.  
\item[(ii)]The fixed points of $A$ and $B$ do not lie on a common $\C^2$-chain.
\item[(iii)]  The fixed-point set of $A$ is disjoint from at least one of the $\C^2$-chains polar to the positive eigenvectors of $B$ and, the fixed-point set of $B$ is disjoint from at least one of the $\C^2$-chains polar to the positive eigenvectors of $A$. 
\end{itemize} 
Condition (iii) can also be stated in terms of the Goldman's invariants introduced in \secref{redu}. It is equivalent to the condition that for some $i, j \in \{1,2\}$, 
$\eta_i(A, B) \neq 0$ and $\nu_j (A, B) \neq 0$, see \secref{mnth}. 

The free subgroup $\langle A, B \rangle$ is \emph{ non-singular} if the generating pair is non-singular.  
In particular,  a non-singular subgroup is Zariski-dense in ${\rm SU}(3,1)$. To a non-singular pair $(A, B)$, we associate complex numbers $\alpha_i(A, B)$ and $\beta_j(A, B)$  given by the following: 
$$\alpha_1(A,B)=\X(\r_A,~\a_A,~\x_B,~\a_B), ~  ~\alpha_2(A,B)=\X(\r_A,~\a_A,~\y_B,~\a_B), $$
$$\beta_1(A,B)=\X(\r_B,~\a_B,~\x_A,~\a_A), ~ ~ \beta_2(A, B)=\X(\r_B,~\a_B,~\y_A,~\a_A),$$
where $\X(\x_1, ~\x_2,~ \x_3,~ \x_4)$ is given by \eqnref{ecr}. 
We shall refer to $\alpha_1(A, B)$ or $\alpha_2(A, B)$ by \emph{$\alpha$-invariant} and, $\beta_1(A, B)$ or $\beta_2(A, B)$ by \emph{$\beta$-invariant}. Condition (iii) in the above definition ensures that  there exist at least one non-zero $\alpha$-invariant and one non-zero $\beta$-invariant for a non-singular subgroup.  We prove the following. 
\begin{theorem} \label{mainth} Let $ A$ and $B $ be two loxodromic elements in 
${\rm SU }(3,1)$ such that they generate a non-singular subgroup $\langle A, B \rangle$. Then 
$\langle A, B \rangle$ is determined up to conjugacy by the following parameters: 

\medskip  $tr(A),~tr(B),~\sigma(A),~\sigma(B)$, 
$~\X_{k}(A, B),~k=1,2,3$,  any one of the non-zero $\alpha$-invariants and any one of the non-zero $\beta$-invariants,  where $tr(A)=\hbox{trace}(A)$, $\sigma(A)=\frac{1}{2}(tr^2(A)-tr(A^2))$.  
  \end{theorem}

 In the parameter space associated to $\langle A, B \rangle$, the parameters ${\rm tr}(A), ~ {\rm tr}(B), ~\alpha$ and $\beta$ are complex numbers, 
${\sigma}(A), ~ {\sigma}(B)$ are real numbers,  
$({\mathbb X}_1, \X_2, \X_3)$ live on the cross-ratio variety that is $5$ real dimensional. Thus we need a total of
$( 4\times 2 + 2\times1 + 5) = 15$ real parameters to specify $\langle A, B\rangle$ up to conjugacy. We call these parameters the \emph{Fenchel-Nielsen coordinates} of $\langle A, B \rangle$.  

\medskip Suppose ${\rm F}_2$ is a free group of rank two. Let ${\rm F}_2=\langle m, n \rangle$. Consider the ${\rm SU}(3,1)$-deformation space of ${\rm F}_2$:  $\mathcal M= {\rm Hom}({\rm F}_2, {\rm SU}(3,1))/{\rm SU}(3,1)$. Let ${\mathcal R}^{lox}$ be the subset of   
${\rm Hom}({\rm F}_2, {\rm SU}(3,1))$ defined by
$${\mathcal R}^{lox}=\{\rho: {\rm F}_2 \to {\rm SU}(3,1) \ | \ \rho(m) \hbox{ and  } \rho(n) \hbox{ are loxodromics}\}.$$
For $i, j \in \{1, 2\}$, let 
$${{\mathcal R}_{ij}^{lox}}=\{\rho \in {\mathcal R}^{lox} \ | \ (\rho(m), \rho(n)) \hbox{ is non-singular s.t. }
 \alpha_i(\rho(m), \rho(n)) \neq 0  \hbox{ and }   \beta_j(\rho(m), \rho(n))\neq 0\}.$$
Let ${{\mathcal R}_{o}^{lox}}=\{\rho \in {\mathcal R}^{lox} \ | \ (\rho(m), \rho(n)) \hbox{ is non-singular}\}$. Clearly,   
$${{\mathcal R}_{o}^{lox}}={{\mathcal R}_{11}^{lox}}\cup {{\mathcal R}_{12}^{lox}}\cup {{\mathcal R}_{21}^{lox}}\cup{{\mathcal R}_{22}^{lox}}.$$
Let ${\mathcal M}_{ij}^{lox} = {\mathcal R}_{ij}^{lox}/{\rm SU}(3,1)$.  Then  \thmref{mainth} classifies the representations  of ${\mathcal M}_{ij}^{lox}$. 

\begin{corollary} 
Let $\rho: {\rm F}_2 \to {\rm SU}(3,1)$ be a representation such that $\rho(m), ~ \rho(n)$ are loxodromics and generate a non-singular subgroup of ${\rm SU}(3,1)$.  Then for some $i, j \in \{1,2\}$, there exist two non-zero complex parameters $\alpha_i$ and $\beta_j$ such that these,  along with the coefficients of the characteristic polynomials  of $\rho(m)$ and $\rho(n)$,  and,  a point on the cross-ratio variety, completely determine $\rho$ up to conjugacy.  The real dimension of the parameter space associated to ${\mathcal M}_{ij}^{lox}$ is 15. 
\end{corollary} 

Let ${\mathcal M}_o^{lox}={{\mathcal R}_{o}^{lox}}/{\rm SU}(3,1)$. It follows that ${\mathcal M}_o^{lox}$ has been covered by four coordinate patches ${\mathcal M}_{ij}^{lox}$ which are parametrized by a subset in $\R^{15}$.  On the intersection of any two of these coordinate patches, we have a transition map that is induced by a linear transformation of $\C^{3,1}$. For example, on ${{\mathcal  M}_{11}^{lox}}\cap {{\mathcal M}_{12}^{lox}}$, 
this transition map is induced by the complex linear transformation that fixes $\aa, \ra, \ab$ and rotates $\xb$ to $\yb$.

\medskip  Let $\mathcal C= \{\gamma_j\}$, $j=1,2,\ldots, 3g-3$,  be a maximal family of simple closed curves on $\Sigma_g$ such that they are neither homotopically equivalent to each-other, nor homotopically trivial.  We may assume that each curve in $\mathcal C$ is actually a geodesic in its homotopy class, and hence, the homotopy type of the curves can be considered as an element in $\pi_1(\Sigma_g)$. We also assume that our curve system is \emph{simple}, i.e., there are $g$ of the curves $\gamma_j$, that correspond to two boundary components of the same three-holed sphere.  We consider discrete, faithful representations $\rho: \pi_1(\Sigma_g) \to {\rm SU}(3,1)$ such that the $3g-3$ group elements $\rho(\gamma_j)$ are loxodromics. Then each pair of pants in the complement of $\mathcal C$ gives rise to a $(0, 3)$ subgroup of ${\rm SU}(3,1)$. The three boundary curves in a pair of pants is represented by $A$, $B$ and 
$B^{-1} A^{-1}$ respectively. For this reason, $A$, $B$ and $B^{-1} A^{-1}$ are called \emph{peripheral elements} of the two-generator subgroup $\langle A, B \rangle$. 

 A discrete, faithful, totally loxodromic representation $\rho: \pi_1(\Sigma_g) \to {\rm SU}(3,1)$ is called \emph{tame} if each of the $(0, 3)$ groups obtained from the given pant decomposition, is non-singular in ${\rm SU}(3,1)$. We aim to describe $30g-30$ parameters that specify tame representations up to conjugacy. 

\begin{theorem}\label{mth2} 

Let $\Sigma_g$ be a closed surface of genus $g$ with a simple curve system $\mathcal C= \{\gamma_j\}$, $j=1,2,\ldots, 3g-3$. Let $\rho: \pi_1(\Sigma_g) \to {\rm SU}(3,1)$ be a tame representation of the surface group $\pi_1(\Sigma_g)$ into ${\rm SU}(3,1)$. Then we need $30g-30$ real parameters to specify $\rho$ in the deformation space  ${\rm  Hom}(\pi_1(\Sigma_g), {\rm SU}(3,1))/{\rm SU}(3,1)$.  

Specifically, these coordinates are $4g-4$ complex traces, $4g-4$ $\sigma$-invariants, $2g-2$ points on the  $5$ (real) dimensional cross-ratio variety, $2g-2$ of the $\alpha$ invariants, $2g-2$ of the $\beta$-invariants and, $3g-3$ twist-bend parameters subject to $3g-3$ constraints corresponding to the compatibility conditions. 
\end{theorem}
The idea to prove the above theorem is similar to the construction of Fenchel-Nielsen coordinates of the classical Teichm\"uller space. The proof follows from \thmref{mainth} combining it with  \propref{tb1} and \propref{tw2} that determine $(0, 4)$ and $(1,1)$ groups respectively.

\medskip This paper is organized as follows.  In \secref{prel},  we recall the notion of the complex hyperbolic space and summarize preliminary results about loxodromic isometries that would be needed later on. In \secref{crai}, we study the Kor\'anyi-Reimann cross-ratios of a quadruple of distinct points on the ideal boundary of $\ch^3$. 
In \secref{redu}, we derive a sufficient condition for the subgroup $\langle A, B \rangle$ to be reducible in terms of the numerical invariants.  We determine the non-singular subgroups $\langle A, B \rangle$ in \secref{mnth}. We prove \thmref{mainth} in this section. In \secref{twb}, we introduce the twist-bend parameters and show that a $(0, 4)$ group  or, an  $(1,1)$ group, is  determined uniquely, up to conjugation, by several invariants introduced in \thmref{mainth}  along with the twist-bend parameters. In particular, we prove  \propref{tb1} and \propref{tw2} in this section.  We prove \thmref{mth2} in \secref{mnth2}.

\section{Preliminaries}\label{prel} 
\subsection{Complex Hyperbolic Space} 
Let $V=\C^{3,1}$ be the complex vector space $\C^4$ equipped with the Hermitian form of signature (3,1) given by $$\langle\z,\w\rangle=\w^{\ast}H\z=z_{1}\bar w_4+z_2\bar w_2+z_3\bar w_3
+z_4\bar w_1,$$
where $\ast$ denotes conjugate transpose. The matrix of the Hermitian form is given by 
\begin{center}
$H=\left[ \begin{array}{cccc}
           0 & 0 & 0 & 1\\
           0 & 1 & 0 & 0\\
 0 & 0 & 1 & 0\\
 1 & 0 & 0 & 0\\
          \end{array}\right].$
\end{center}
If $H^{\prime}$ is any other $4\times4$ Hermitian matrix with signature (3,1), then there is a $4\times4$ matrix $C$ so that $C^{\ast}H^{\prime}C=H$.

We consider the following subspaces of $\C^{3,1}:$
$$V_{-}=\{\z\in\C^{3,1}:\langle\z,\z \rangle<0\}, ~ \V_+=\{\z\in\C^{3,1}:\langle\z,\z \rangle>0\},$$
$$V_{0}=\{\z-\{\bf 0\}\in\C^{3,1}:\langle\z,\z \rangle=0\}.$$
A vector $\z$ in $\C^{3,1}$ is called \emph{positive, negative}  or \emph{null}  depending on whether $\z$ belongs to $V_+$,   $V_-$ or  $V_0$. Let $\P:\C^{3,1}-\{0\}\longrightarrow\C P^3$ be the canonical projection onto complex projective space. The complex hyperbolic space $\ch^3$ is defined to be $\P V_{-}$. The ideal boundary $\partial\ch^3$ is $\P V_{0}$. The canonical projection of a vector $\z$ is given by 
\hbox{$z=\P(\z)=(z_1/z_4,z_2/z_4,z_3/z_4)$}. Therefore we can write $\ch^3=\P(V_{-})$ as
$$\ch^3=\{(w_1,w_2,w_3)\in\C^3 \ : \ 2\Re(w_1)+|w_2|^2+|w_3|^2<0\}.$$  
This gives the Siegel domain model of $\ch^3$. 
There are two distinguished points in $V_{0}$ which we denote by  ${o}$ and $\bf\infty$ given by 
$$\bf{o}=\left[\begin{array}{c}
               0\\0\\0\\1\\
              \end{array}\right],
 \infty=\left[\begin{array}{c}
               1\\0\\0\\0\\
              \end{array}\right].$$\\
Then we can write $\partial\ch^3=\P(V_{0})$ as
$$\partial\ch^3-\infty=\{(z_1,z_2,z_3)\in\C^3:2\Re(z_1)+|z_2|^2+|z_3|^2=0\}.$$\\
 Conversely, given a point $z$ of $\ch^3=\P(V_{-})\subset\C P^3$ we may lift $z=(z_1,z_2,z_3)$ to a point $\z$ in $V_{-}$, called the standard lift of $z$, 
 by writing in non-homogeneous coordinates as
 $$\z=\left[\begin{array}{c}
                z_1\\z_2\\z_3\\1\\
               \end{array}\right].$$
 The Bergman metric in $\ch^3$ is defined in terms of the Hermitian form given by:
$${ds}^2=-\frac{4}{\langle z,z \rangle^2} \det \left[\begin{array}{cc}
                                                    \langle z,z \rangle & \langle dz ,z \rangle\\
                                                    \langle z,dz \rangle & \langle dz,dz \rangle\\ 
                                                   \end{array}\right].$$
If $z$ and $w$ in $\ch^3$ correspond to vectors $\z$ and $\w$ in $V_{-}$, then the Bergman metric is given by the distance $\rho$: 
$$\cosh^2\big(\frac{\rho(z,w)}{2}\big)=\frac{\langle\z,\w\rangle \langle\w,\z\rangle}{\langle\z,\z\rangle \langle\w,\w\rangle}.$$ 
\subsection{Isometries}   
 Let ${\rm U}(3,1)$ be the isometry group of  the Hermitian form $\langle .,.\rangle$. Each matrix $A$ in ${\rm U}(3,1)$ satisfies the relation $A^{-1}=
H^{-1}A^{\ast}H$, where $A^{\ast}$ is the conjugate transpose of $A$. The holomorphic isometry group of  $\ch^3$ is the projective unitary group ${\rm PSU}(3,1)={\rm SU }(3,1)/\{\pm I,\pm iI\}$.  It is often more convenient to lift to the four-fold covering ${\rm SU }(3,1)$ and view  the linear action of the isometries. 

\medskip Based on their fixed points, holomorphic isometries of $\ch^3$ are classified as follows:
\begin{enumerate}
\item An isometry is \emph{elliptic} if it fixes at least one point of $\ch^3$.
\item An isometry is \emph{parabolic} if it is non-elliptic and fixes exactly one point of $\partial\ch^3$.
\item An isometry is \emph{loxodromic} if it is non-elliptic and fixes exactly two points of $\partial\ch^3$, one of which is attracting and other repelling.  
\end{enumerate}
For algebraic classification of isometries of $\ch^3$ using their conjugacy invariants, see \cite{gpp, g}.
\subsection{Loxodromic Isometries} \label{li}
Let $A\in {\rm SU }(3,1)$ represents a loxodromic isometry.  Then $A$ has eigenvalues
 of the form $re^{i\theta},~r^{-1}e^{i\theta},~e^{i\phi},~e^{-i(2\theta+\phi)}$, $\theta, \phi \in (-\pi, \pi]$.  Let $a_A\in\partial\ch^3$ be the attractive fixed point of $A$, then any lift ${\bf a}_A$ of $a_A$ to $V_0$ is an eigenvector of $A$ and corresponds to the 
eigenvalue  $re^{i\theta}$. Similarly,  if $r_A\in\partial\ch^3$ is the repelling fixed point of $A$, then any lift $\bf{r}_A$ of $r_A$ to $V_0$ is an
eigenvector of $A$ with eigenvalue $r^{-1}e^{i\theta}.$\\ 

For $(r,\theta,\phi)\in S$, define $E(r,\theta,\phi)$ as\\
\begin{equation}\label{li1}
 E(r,\theta,\phi)=\left[\begin{array}{cccc}
                         re^{i\theta} & ~ &~ & ~\\
                         ~ & e^{i\phi} & ~ & ~\\
                         ~ & ~ & e^{-i(2\theta+\phi)} & ~\\
                         ~ & ~ & ~ & r^{-1}e^{i\theta}\\
                        \end{array}\right].\\
\end{equation}
It is easy to see that $E=E(r,\theta,\phi)\in {\rm SU }(3,1)$ represents  a loxodromic map with attractive fixed point $a_E=\bf\infty$ and repelling fixed point 
$r_E={o}$. Equivalently, \eqnref{li1} can be represented in the form 
$$E(\lambda, \psi)={\rm diag}(e^{\lambda}, \ e^{-i \psi + \frac{\overline \lambda -\lambda}{2}}, \ e^{i \psi + \frac{\overline \lambda -\lambda}{2}}, \ e^{-\overline \lambda}), \ \lambda=l + i \theta. $$
We may also assume that  $r>1$ and $\theta, \phi \in (-\pi, \pi]$, and thus  $(\lambda, \psi)\in S$, where $S$ is the region given by:
$$S=\{(\lambda, \psi) \ : \ \Re(\lambda)>0, ~ \Im(\lambda), \psi \in (-\pi, \pi]\}.$$
Let $\textbf{x}_A,~\textbf{y}_A$ be the eigenvectors corresponding to the eigenvalues $e^{i\phi},~e^{-i(2\theta+\phi)}$ respectively, scaled so that 
$\langle \textbf{x}_A,\textbf{x}_A \rangle=1=\langle \textbf{y}_A,\textbf{y}_A \rangle$. Let $C_A=\left[\begin{array}{cccc}
 \textbf{a}_A & \textbf{x}_A & \textbf{y}_A & \textbf{r}_A\\
  \end{array}\right]$ be the $4\times4$ matrix, where
the lifts $\bf{a}_A$ and $\bf{r}_A$ are chosen so that $C_A$ has determinant $1$. Then $C_A\in {\rm SU }(3,1)$ and $A=C_{A}E_{A}(r,\theta,\phi)C_{A}^{-1}$, where
$E_{A}(r,\theta,\phi)$ is given by \eqnref{li1}.
\begin{lemma}\label{lile1}
 Let $A \in {\rm SU }(3,1)$. Then $A$ has characteristic polynomial 
$$\chi_A(X)=X^4-\tau_{A} X^3+\sigma_{A} X^2-\bar\tau_{A} X+1,$$ where $\tau_{A}=tr(A)$ and $\sigma_{A}=\frac{1}{2}(tr^2(A)-tr(A^2))$.
 Moreover $\sigma_{A}$  is real.
\end{lemma}
For a proof see \cite{gp, gl}. We also denote $\sigma_{A}$ by $\sigma(A)$ in the sequel. 
\begin{prop}\label{lipr1}
 Two loxodromic elements in ${\rm SU}(3,1)$ are conjugate if and only if they have the same eigenvalues.
\end{prop}
For a proof,  see \cite{chen}.
An immediate consequence of Lemma \ref{lile1} and \propref{lipr1} is the following. 
\begin{corollary}\label{licor1}
 Two loxodromic elements $A$ and $A^\prime$ in ${\rm SU }(3,1)$ are conjugate if and only if $\tau_{A}=\tau_{A^\prime}$ and $\sigma_A=\sigma_{A^\prime}$.
\end{corollary}
\begin{lemma}\label{lile2}
 Let $A=\begin{bmatrix} Ae_1 & Ae_2 & Ae_3 & Ae_4 \end{bmatrix} \in {\rm SU }(3,1)$.  The vector $Ae_3$ is uniquely determined by the vectors $Ae_1,~Ae_2$ and $Ae_4$.
\end{lemma}
\begin{proof}
 Let $W$ be the subspace spanned by $Ae_1, Ae_2, Ae_4$.
Let  $W^{\bot}$ be the orthogonal complement of $W$ in $\C^{3,1}$. 
Observe that since $A\in {\rm SU }(3,1),~W \cap W^\bot=\{0\}$ and $W^\bot\neq\{0\}$ is an one dimensional subspace of $\C^4$. Let $W^{\bot}=\langle w \rangle$ for some
$w\in \C^4 $. Then $Ae_3 \in W^{\bot}$ implies that $Ae_3= \lambda w$ for some $\lambda\in \C$. Further the condition $\det(A)=1$ determines $\lambda$ uniquely and the 
assertion follows. 
\end{proof}
\begin{corollary}\label{licor2}
Let $A=\begin{bmatrix} Ae_1 & Ae_2 & Ae_3 & Ae_4\end{bmatrix} ,~B=\begin{bmatrix}Be_1 & Be_2 & Be_3 & Be_4\end{bmatrix} \in {\rm SU }(3,1)$ and $C\in {\rm SU }(3,1)$ be such that $CAe_i=Be_i$ for $i=1,~2,~4$, then $CAe_3=Be_3.$
\end{corollary}
From \lemref{lile2} and \corref{licor1},  we have the following.
\begin{corollary}\label{licor3}
 Let $A$ and $A'$ are two loxodromic elements in ${\rm {\rm SU }}(3,1)$ such that $\tau_{A}=\tau_{A^\prime}$, $\sigma_A=\sigma_{A^\prime},~a_{A}=a_{A'},~
r_{A}=r_{A'}$ and $x_{A}=x_{A'}$, then $A=A'$.  
\end{corollary}

\section{The Cross-Ratios}\label{crai}

\subsection{The Kor\'anyi-Reimann cross-ratio}\label{cr}
Given a quadruple of distinct points $(z_1, z_2, z_3, z_4)$ on $\partial \ch^3$, their Kor\'anyi-Reimann cross-ratio is defined by
$$\X(z_1, z_2, z_3, z_4)=[z_1,z_2,z_3,z_4]=\frac{\langle {\bf z}_3, {\bf z}_1 \rangle \langle {\bf z}_4, \bf z_2 \rangle} { \langle {\bf z}_4, {\bf z}_1\rangle \langle   {\bf z}_3, {\bf z}_2 \rangle},$$
where, for $i=1,2,3,4$,  ${\bf z}_i$, are lifts of $z_i$. It can be seen easily that $\X$ is independent of the chosen lifts of $z_i$'s. By choosing different ordering of the four points, we may define other cross-ratios and it can be seen in \cite[p.225]{gold} that there are certain symmetries that are associated with certain permutations. After taking these into account, there are only three cross-ratios that remain. Given distinct points $z_1,~z_2,~z_3,~z_4$ in $\partial\ch^3$, we define :
\begin{equation}\label{crveq1}
 \X_1=[z_1,~z_2,~z_3,~z_4],~ \X_2=[z_1,~z_3,~z_2,z_4],~ \X_3=[z_2,~z_3,~z_1,~z_4]. 
\end{equation}
 Parker and Platis \cite{pp}, also see Falbel \cite{f},  have shown that the triples of cross-ratios of an ordered quadruple of points in $\partial\ch^2$ satisfy two real equations. If the ordered triples of points belongs to $\partial\ch^3$, the corresponding cross-ratios satisfy only one real equation and one real inequality as shown in 
the following proposition. 
 
\begin{prop}\label{crp}
Let $z_1,~z_2,~z_3,~z_4$ be four distinct points in $\partial\ch^3$. Let $\X_1,~\X_2,~\X_3$ be defined by \ref{crveq1}, then 
\begin{equation} \label{cr1} |\X_2|=|\X_1||\X_3|,\end{equation}
 \begin{equation} \label{cr2} 2|\X_1|^2\Re(\X_3)\ge|\X_1|^2+|\X_2|^2+1-2\Re(\X_1+\X_2). \end{equation}	
Further, equality holds in \eqnref{cr2} if and only if either of the following holds: 
\begin{enumerate}
\item[(i)] $z_1, ~ z_2, ~ z_4$ lie on the same complex line. 
\item[(ii)] $z_1, ~ z_3, ~ z_4$ lie on the same complex line. 
\item[(iii)] $z_1, ~ z_2, ~ z_3, ~ z_4$ lie on the same complex line. 
\end{enumerate}
\end{prop}

\begin{proof}
 Since ${\rm SU }(3,1)$ acts doubly transitively on $\partial\ch^3$, we may suppose that  $z_2=\bf\infty$ and $z_3={o}$. Let $\z_1, \z_4$ be lifts of $z_1$ and
$z_4$ chosen so that $\langle \z_1, \z_4 \rangle=1$. We write them in coordinates as 
$$\z_1=\left[\begin{array}{c}
             \xi_1\\
             \eta_1\\
             \alpha_1\\   
             \delta_1\\ 
            \end{array}\right], 
~\z_2=\left[\begin{array}{c}
             1\\0\\0\\0\\
                \end{array}\right],
 ~\z_3=\left[\begin{array}{c}
             0\\0\\0\\1\\
            \end{array}\right],
 ~\z_4=\left[\begin{array}{c}
             \xi_4\\
             \eta_4\\
             \alpha_4\\   
             \delta_4\\ 
            \end{array}\right].$$
Then we have
\begin{equation}\label{1}
 0=\langle \z_1,~\z_1 \rangle=\xi_1\bar \delta_1+\bar \xi_1\delta_1+|\eta_1|^2+|\alpha_1|^2; 
\end{equation}
\begin{equation}\label{2}
 1=\langle \z_4,~\z_1 \rangle=\xi_4\bar \delta_1+\delta_4\bar \xi_1+\eta_4\bar \eta_1+\alpha_4\bar \alpha_1; 
\end{equation}
\begin{equation}\label{3}
 0=\langle \z_4,~\z_4 \rangle=\xi_4\bar \delta_4+\bar \xi_4\delta_4+|\eta_4|^2+|\alpha_4|^2. 
\end{equation}
From the definitions of the cross-ratios we have:
$$\X_1=[z_1,~z_2,~z_3,~z_4]=\frac{\langle \z_3,~\z_1\rangle \langle \z_4,~\z_2\rangle}{\langle \z_4,~\z_1\rangle \langle \z_3,~\z_2\rangle}=
\bar \xi_1\delta_4, $$\\
$$\X_2=[z_1,~z_3,~z_2,~z_4]=\frac{\langle \z_2,~\z_1\rangle \langle \z_4,~\z_3\rangle}{\langle \z_4,~\z_1\rangle \langle \z_2,~\z_3\rangle}=
\xi_4\bar\delta_1, $$\\
$$\X_3=[z_2,~z_3,~z_1,~z_4]=\frac{\langle \z_1,~\z_2\rangle \langle \z_4,~\z_3\rangle}{\langle \z_4,~\z_2\rangle \langle \z_1,~\z_3\rangle}=
\frac{\xi_4\delta_1}{\xi_1\delta_4}.$$
We immediately see that $$|\X_3|=\frac{|\X_2|}{|\X_1|}.$$\\
Using equtaions \ref{1}--\ref{3}, we have:\\
\begin{align*}
 |\X_1|^2|\X_3-1|^2&=|\xi_4\delta_1-\xi_1\delta_4|^2\\
&=|\xi_4\delta_1|^2+|\xi_1\delta_4|^2-\xi_4\delta_1\bar\xi_1\bar\delta_4-\bar\xi_4\bar\delta_1\xi_1\delta_4\\
&=|\bar \xi_1\delta_4|^2+|\xi_4\bar \delta_1|^2+\xi_4\bar\delta_4(\xi_1\bar\delta_1+|\eta_1|^2+|\alpha_1|^2)+\bar\xi_4\delta_4
(\bar\xi_1\delta_1+|\eta_1|^2+|\alpha_1|^2)\\
&=|\bar \xi_1\delta_4+\xi_4\bar\delta_1|^2-(|\eta_4|^2+|\alpha_4|^2)(|\eta_1|^2+|\alpha_1|^2)\\
&=|\bar \xi_1\delta_4+\xi_4\bar\delta_1|^2-|\eta_4\bar\eta_1+\alpha_4\bar\alpha_1|^2-|\eta_4\alpha_1-\eta_1\alpha_4|^2\\
&=|\X_1+\X_2|^2+|1-\X_1-\X_2|^2-|\eta_4\alpha_1-\eta_1\alpha_4|^2. 
\end{align*}
This implies, 
$$|\X_1|^2|\X_3-1|^2-|\X_1+\X_2|^2+|1-\X_1-\X_2|^2=-|\eta_4\alpha_1-\eta_1\alpha_4|^2\le0.$$
Rearranging this gives the inequality we want. Further the above inequality is an equality if and only if 
$$\eta_4\alpha_1-\eta_1\alpha_4=0, \ i.e., \ \ \frac{\eta_4}{\eta_1}=\frac{\alpha_4}{\alpha_1}.$$
This means either of the conditions (i), (ii) , (iii) given in the statement. This proves the proposition. 
\end{proof}
Platis \cite{pl} has proved a generalization of the above proposition for arbitrary rank-one symmetric spaces of non-compact type and has applied it to  derive Ptolemaean inequality on the boundary of a rank-one symmetric space of non-compact type. Since we have restricted ourselves only to three dimensional complex hyperbolic geometry, our proof above is much simpler. 
\begin{corollary}
 Let $\X_1,~\X_2,~\X_3$ be defined by \ref{crveq1}, then $2\Re(\X_1+\X_2)\ge1$.
\end{corollary}
\begin{proof}
 Since $\Re(\X_3)\le|\X_3|,$
\begin{align*}
2\Re(\X_1+\X_2)-1&\geq|\X_1|^2+|\X_2|^2-2|\X_1|^2\Re(\X_3)\\
&\geq|\X_1|^2+|\X_2|^2-2|\X_1|^2|\X_3|\\
&=|\X_1|^2+|\X_2|^2-2|\X_1||\X_2|\\
&=(|\X_1|-|\X_2|)^2\geq0. 
\end{align*}
In particular $2\Re(\X_1+\X_2)\geq1$. 
\end{proof}

\subsection{Cartan's angular invariant}\label{cai}
Let $z_1,~z_2,~z_3$ be three distinct points of $\partial\ch^3$ with lifts $\z_1,~\z_2$ and $\z_3$ respectively. Cartan's angular invariant is defined as 
follows :
$$\A(z_1,~z_2,~z_3)=arg(-\langle \z_1,\z_2 \rangle \langle \z_2,\z_3 \rangle \langle \z_3,\z_1 \rangle).$$
The angular invariant is invariant under ${\rm SU }(3,1)$ and independent of the chosen lifts. The following proposition shows that this invariant determines 
any triples of distinct points in $\partial\ch^3$ up to ${\rm SU }(3,1)$-equivalence.
\begin{prop}\label{cai1}
Let $z_1,~z_2,~z_3$ and $z_1^\prime,~z_2^\prime,~z_3^\prime$ be triples of distinct points of $\partial\ch^3$. Then $\A(z_1,~z_2,~z_3)=\A(z_1^\prime,~
z_2^\prime,~z_3^\prime)$ if and only if there exist $A\in {\rm SU }(3,1)$ so that $A(z_j)=z_j^\prime$ for $j=1,2,3$. 
\end{prop}
For a proof see \cite{gold}. Also we have the following result from \cite{gold}.
\begin{prop}\label{cra}
Let $z_1$, $z_2$, $z_3$ be three distinct points of $\partial \ch^3$ and let $\A=\A(z_1, z_2, z_3)$ be their angular invariant. Then
\begin{enumerate}
\item $\A \in [-\frac{\pi}{2}, \frac{\pi}{2}]$. 
\item $\A=\pm \frac{\pi}{2}$ if and only if $z_1$, $z_2$, $z_3$ lie on the same chain.
\item $\A=0$ if and only if $z_1$, $z_2$, $z_3$ lie on a totally real totally geodesic subspace. 
\end{enumerate}
\end{prop}

\begin{prop}\label{crvpr2}
 Let $z_1,~z_2, ~z_3,~z_4$ be distinct points of $\partial\ch^3$ and let $\X_1,~\X_2,~\X_3$ denote the cross-ratios defined by \ref{crveq1}. Suppose $\X_1$, $\X_2$ and $\X_3$ are non-real complex numbers. Let $\A_1=
\A(z_4,~z_3,~z_2)$ and $\A_2=\A(z_3,~z_2, z_1)$. Then 
\begin{enumerate}
 \item $\A_1 +\A_2=arg(\overline\X_1\X_2)$.
\item $\A_1 -\A_2=arg(\X_3)$.
\end{enumerate}
\end{prop}
Note that the above proposition is not true if $\X_i$'s are real numbers.  Cuhna-Gusevskii \cite[p.279]{gc} have given a counter-example to the above proposition when $\X_i$'s are real numbers.  
However, when all the cross-ratios are non-real, the argument as in the proof of \cite[Proposition 5.8]{pp} goes through and we have the above proposition. An explanation that the proof of \cite[Proposition 5.8]{pp} does not carry over to the real cross-ratio case is that the principal argument of complex numbers is a well-defined function from $\C-\{0\}$ to the semi-open interval $(-\pi, \pi]$. On the other hand, $\A_1 \pm \A_2$ are well-defined functions from distinct triple points on $\partial \ch^n$ onto the closed interval $[-\pi, \pi]$. So, the principal argument can not be identified with $\A_1 \pm \A_2$, especially on the boundary points of the intervals and those cases correspond when the cross-ratios are real numbers. 

\begin{prop}\label{crvpr3}
 Let $z_1,~z_2,~z_3,~z_4$ be distinct points of $\partial\ch^3$ with  non-real cross-ratios $\X_1,~\X_2,~\X_3$. Let $z_1^\prime,~z_2^\prime,~z_3^\prime,~z_4^\prime$
be another set of distinct points of $\partial\ch^3$ with corresponding cross-ratios $\X_1^ \prime,~\X_2^\prime,~\X_3^\prime$. If $\X_i^\prime=\X_i$ for 
$i=1,~2,~3$, then there exist $A\in {\rm SU }(3,1)$ such that $A(z_j)=z_j^\prime$ for $j=1,~2,~3,~4$.    
\end{prop}

\begin{proof}
 Since ${\rm SU }(3,1)$ acts doubly transitively on $\partial\ch^3$, we may assume without loss of generality that $z_2=z_2^\prime=\infty,~  z_3=z_3^\prime=o$. We write the lifts of 
other points as 

$$\z_1=\left[\begin{array}{c}
             \xi_1\\
             \eta_1\\
             \alpha_1\\   
             \delta_1\\ 
            \end{array}\right], 
~\z_4=\left[\begin{array}{c}
             \xi_4\\
             \eta_4\\
             \alpha_4\\   
             \delta_4\\ 
            \end{array}\right].
~\z_1^\prime=\left[\begin{array}{c}
             \xi_1^\prime\\
             \eta_1^\prime\\
             \alpha_1^\prime\\   
             \delta_1^\prime\\ 
            \end{array}\right],
~\z_4^\prime=\left[\begin{array}{c}
             \xi_4^\prime\\
             \eta_4^\prime\\
             \alpha_4^\prime\\   
             \delta_4^\prime\\ 
            \end{array}\right].$$
We may suppose that lifts of these points are chosen so that $\langle \z_4,\z_1 \rangle=\langle \z_4^\prime,\z_1^\prime \rangle$, i.e $$\bar\xi_1\delta_4+\xi_4\bar\delta_1+\eta_4\bar\eta_1+\alpha_4\bar\alpha_1=\bar\xi_1^\prime\delta_4^\prime+\xi_4^\prime\bar\delta_1^\prime
+\eta_4^\prime\bar\eta_1^\prime+\alpha_4^\prime\bar\alpha_1^\prime.$$
Then our condition on the cross-ratios gives :
\begin{align*}
\bar\xi_1\delta_4&=\bar\xi_1^\prime\delta_4^\prime, \\ 
\xi_4\bar\delta_1&=\xi_4^\prime\bar\delta_1^\prime, \\
\frac{\xi_4\delta_1}{\xi_1\delta_4}&=\frac{\xi_4^\prime\delta_1^\prime}{\xi_1^\prime\delta_4^\prime}.
\end{align*}
Hence we also have
\begin{equation}\label{crv1}
 \eta_4\bar\eta_1+\alpha_4\bar\alpha_1=\eta_4^\prime\bar\eta_1^\prime+\alpha_4^\prime
\alpha_1^\prime.
\end{equation}
Let us denote the angular invariants of the points by $\A_1=\A(z_4,~z_3,~z_2),~\A_2=\A(z_3,~z_2,~z_1)$,  $\A_1^\prime=\A(z_4^\prime,~z_3^\prime,~z_2^\prime)
,~\A_2^\prime=\A(z_3^\prime,~z_2^\prime,~z_1^\prime)$. Using \propref{crvpr2}, we see that $\A_1+\A_2=\A_1^\prime+\A_2^\prime$ 
and $\A_1-\A_2=\A_1^\prime-\A_2^\prime$. 
Hence $~\A_1=\A_1^\prime$ and $\A_2=\A_2^\prime$. From \propref{cai1}, we see that there exist $A_1,~A_2\in {\rm SU }(3,1)$ such that
 $A_1(z_2)=z_2^\prime,A_1(z_3)=z_3^\prime,~A_1(z_4)=z_4^\prime$ and $A_2(z_1)=z_1^\prime,~A_2(z_2)=z_2^\prime,~A_2(z_3)=z_3^\prime$.\\
Because $A_1$ fixes $z_2=\bf\infty$ and $z_3=0$, it is of form
$$\left[\begin{array}{ccc}
       \lambda &~ & ~\\
  ~ & U_1 &~\\
 ~& ~ &\bar\lambda^{-1} \\
      \end{array}\right],$$
where, $|\lambda|\neq 1$ and $U_1\in {\rm U}(2)$. Hence we have
 $\lambda\xi_4=\xi_4^\prime,~\bar\lambda^{-1}\delta_4=\delta_4^\prime$ and
 \hbox{$U_1\left[\begin{array}{c}
             \eta_4\\
             \alpha_4\\
            \end{array}\right]=
\left[\begin{array}{c}
             \eta_4^\prime\\
             \alpha_4^\prime\\
            \end{array}\right]$.}
Therefore
\begin{align*}
\xi_1^\prime&=\frac{\bar\delta_4}{\bar\delta_4^\prime}\xi_1\\
&=\lambda\xi_1,\\
\delta_1^\prime&=\frac{\bar\xi_4}{\bar\xi_4^\prime}\delta_1\\
&=\bar\lambda^{-1}\delta_1.
\end{align*}
Hence $A_2$ is of form
$$\left[\begin{array}{ccc}
       \lambda &~ & ~\\
  ~ & U_2 &~\\
 ~& ~ &\bar\lambda^{-1} \\
      \end{array}\right],$$ 
where, $U_2\in {\rm U}(2)$ so that $$U_2\left[\begin{array}{c}
             \eta_1\\
             \alpha_1\\
            \end{array}\right]=
\left[\begin{array}{c}
             \eta_1^\prime\\
             \alpha_1^\prime\\
            \end{array}\right].$$
It is enough to prove that there exist $U\in {\rm U}(2)$ such that
$$U\left[\begin{array}{c}
             \eta_4\\
             \alpha_4\\
            \end{array}\right]=
\left[\begin{array}{c}
             \eta_4^\prime\\
             \alpha_4^\prime\\
            \end{array}\right]  \hbox{ and }
U\left[\begin{array}{c}
             \eta_1\\
             \alpha_1\\
            \end{array}\right]=
\left[\begin{array}{c}
             \eta_1^\prime\\
             \alpha_1^\prime\\
            \end{array}\right].$$

Suppose,   $\bf{y_1}=\left[\begin{array}{c}
             \eta_1\\
             \alpha_1\\
            \end{array}\right],~\bf{y_4}=\left[\begin{array}{c}
             \eta_4\\
             \alpha_4\\
            \end{array}\right],~\bf{y_1}^\prime=\left[\begin{array}{c}
             \eta_1^\prime\\
             \alpha_1^\prime\\
            \end{array}\right],
~\bf{y_4}^\prime=\left[\begin{array}{c}
             \eta_4^\prime\\
             \alpha_4^\prime\\
            \end{array}\right].$\\
From \eqnref{crv1}, we have 
\begin{equation}\label{crv2}
 \ll \bf{y_4},~\bf{y_1}\gg=\ll \bf{y_4}^\prime,~\bf{y_1}^\prime\gg,
\end{equation}
where, $\ll .,. \gg$ is the standard positive-definite Hermitian form on $\C^2$. Also we have $U_1 \bf{y_4}=\bf{y_4}^\prime$ and $U_2\bf{y_1}=\bf{y_1}^\prime$. 
Then, $U_1,~U_2\in {\rm U}(2)$ implies,\\
\begin{equation}\label{crv3}
 \ll \bf{y_4},~\bf{y_4} \gg=\ll \bf{y_4}^\prime,~\bf{y_4}^\prime\gg,
\end{equation}
\begin{equation}\label{crv4}
 \ll \bf{y_1},~\bf{y_1} \gg = \ll \bf{y_1}^\prime,~\bf{y_1}^\prime\gg.\\
\end{equation}
Suppose, $\bf{y_1}$ and $\bf{y_4}$ are linearly independent over $\C$, thus, forming a basis of $\C^2$. Let $U$ be the $2\times 2$ matrix so that
$U\bf{y_1}=\bf{y_1}^\prime$ and $U\bf{y_4}=\bf{y_4}^\prime$. Then from  \eqnref{crv2} -- \eqnref{crv4},  it follows that $U$ preserves the Hermitian form $\ll.,.\gg$ 
on $\C^2$, so $U\in {\rm U}(2)$ and we are done.

Now, consider the case when $\bf{y_1}$ and $\bf{y_4}$ are linearly dependent over $\C$, i.e., $\bf{y_4}=\mu\bf{y_1}$ for some $\mu\in\C$. Since the form
$\ll .,.\gg$ is positive definite, using \ref{crv2}$-$\ref{crv4}, this is true if and only if 
\begin{align*}
&\ll \bf{y_4}-\mu\bf{y_1},\bf{y_4}-\mu\bf{y_1} \gg=0\\
&\Leftrightarrow ~\ll \bf{y_4}^\prime-\mu\bf{y_1}^\prime,\bf{y_4}^\prime-\mu\bf{y_1}^\prime \gg=0\\
&\Leftrightarrow ~\bf{y_4}^\prime=\mu\bf{y_1}^\prime.
\end{align*}
Therefore, either of $U_1$ and $U_2$ works. This completes the proof.    
\end{proof}

\subsubsection{When cross-ratios are all real}  Suppose, all the three cross-ratios are real. Then \eqnref{cr1} implies $\X_3=\pm \X_2/\X_1$. The following result can be proved along the same line as in the proof of \cite[Proposition 5.12]{pp}. 
\begin{lemma}
Suppose $\X_1$, $\X_2$ and $\X_3$ are all real. 
\begin{enumerate}
\item If $\X_3=-\X_2/\X_1$, then the points $z_j$ all lie on a chain.
\item If $\X_3=\X_2/\X_1$, then the points $z_j$ all lie in a totally real Lagrangian subspace. 
\end{enumerate}
\end{lemma} 
The following result follows from \cite[p.225]{gold}. 
\begin{lemma}\label{gl}
Suppose $z_1$, $z_2$, $z_3$ and $z_4$ all lie on the same chain. Then $\X_1$, $\X_2$ and $\X_3$ are each real.  
\end{lemma}

\begin{lemma}
If $z_1$, $z_2$, $z_3$, $z_4$ are contained in the same totally real totally geodesic subspace, then $\X_1$, $\X_2$. $\X_3$ are real numbers. 
\end{lemma}
\begin{proof}
Let $\iota$ be the anti-holomorphic involution fixing the totally real totally geodesic subspace. Then for $i=1,2,3$, applying $\iota$ we get $\X_i=\overline \X_i$. Hence all the cross-ratios are real. 
\end{proof} 

Summarizing the above lemmas we have the following.
\begin{prop}\label{inva}
Let $z_1$, $z_2$, $z_3$, $z_4$ be distinct points on $\partial \ch^3$. Then the cross-ratios $\X_1$, $\X_2$ and $\X_3$ are real numbers if and only if $z_1$, $z_2$, $z_3$, $z_4$ all lie on the same chain or the same totally  real totally geodesic subspace. 
\end{prop}

\section{A sufficient condition for Irreducibility}\label{redu}

Let $A,~B$ be loxodromic elements in ${\rm{{\rm SU }}}(3,1)$ and following the notation of \secref{li}, let $$C_A=\left[\begin{array}{cccc}  \textbf{a}_A & \textbf{x}_A & \textbf{y}_A & \textbf{r}_A\\
                                                                                        \end{array}\right],~ C_B=\left[\begin{array}{cccc}
                                                                                         \textbf{a}_B & \textbf{x}_B & \textbf{y}_B & \textbf{r}_B\\
                                                                                        \end{array}\right]$$ be the eigen matrices associated with $A$  and $B$ respectively. The Kor\'anyi-Reimann cross-ratios of $A$ and $B$ are defined by \\
\begin{equation}
\X_1(A,~B)=[a_B,~a_A,~r_A,~r_B]= \frac{\langle {\bf r}_A,{\bf a}_B \rangle \langle {\bf r}_B,{\bf a}_A \rangle}
{\langle {\bf r}_B,{\bf a}_B \rangle \langle {\bf r}_A,{\bf a}_A \rangle},\\
\end{equation}
\begin{equation}
\X_2(A,~B)=[a_B,~r_A,~a_A,~r_B]= \frac{\langle {\bf a}_A,{\bf a}_B \rangle \langle {\bf r}_B,{\bf r}_A \rangle}
{\langle {\bf r}_B,{\bf a}_B \rangle \langle {\bf a}_A,{\bf r}_A \rangle},\\
\end{equation}
\begin{equation}
\X_3(A,~B)=[a_A,~r_A,~a_B,~r_B]= \frac{\langle {\bf a}_B,{\bf a}_A \rangle \langle {\bf r}_B,{\bf r}_A \rangle}
{\langle {\bf r}_B,{\bf a}_A \rangle \langle {\bf a}_B,{\bf r}_A \rangle}.\\ 
\end{equation}
 In \cite{gold} , Goldman defines $\eta$-invariant for a triple of points with two points on $\partial\ch^{3}$ and one point on $\P(V_{+})$.  Following Goldman's definition, we define
$\eta$-invariants associated to $A$ and $B$ as follows\\ 
\begin{equation*}
\eta_{1}(A,B)= \eta(a_A,r_A;x_B)=\frac{\langle {\bf a}_A,{\bf x}_B \rangle \langle {\bf x}_B,{\bf r}_A \rangle}
{\langle {\bf a}_A,{\bf r}_A \rangle \langle {\bf  x}_B, {\bf x}_B\rangle},\\
\end{equation*}

\begin{equation*}
\eta_{2}(A,B)= \eta(a_A,r_A;y_B)=\frac{\langle {\bf a}_A,{\bf y}_B \rangle \langle {\bf y}_B,{\bf r}_A \rangle}
{\langle {\bf a}_A,{\bf r}_A \rangle \langle {\bf y}_B, {\bf y}_B\rangle},\\
\end{equation*}
\begin{equation*}
\nu_{1}(A,B)= \eta(a_B,r_B;x_A)=\frac{\langle {\bf a}_B,{\bf x}_A \rangle \langle {\bf x}_A,{\bf r}_B \rangle}
{\langle {\bf a}_B,{\bf r}_B \rangle \langle {\bf x}_A, {\bf x}_A \rangle},\\
\end{equation*}
\begin{equation*}
\nu_{2}(A,B)= \eta(a_B,r_B;y_A)=\frac{\langle {\bf a}_B,{\bf y}_A \rangle \langle {\bf y}_A,{\bf r}_B \rangle}
{\langle {\bf a}_B,{\bf r}_B \rangle \langle {\bf y}_A, {\bf y}_A \rangle}.\\
\end{equation*}

We define 
\begin{equation*}
\zeta_o(A,B)=[y_A,~x_A,~x_B,~y_B]= \frac{\langle {\bf x}_B,{\bf y}_A \rangle \langle {\bf y}_B,{\bf x}_A \rangle}
{\langle {\bf x}_B,{\bf x}_A \rangle \langle {\bf y}_B,{\bf y}_A \rangle}.\\ 
\end{equation*}

It is clear from the definition that the $\X_{i}$'s, $\eta_{j}$'s and $\zeta_o$ are conjugacy invariants for the two generator subgroup
 $\langle A,B \rangle$ of ${\rm{{\rm SU }}}(3,1)$ and their values are independent of the chosen lifts of the eigenvectors. 
\begin{theorem}\label{redut}
Let $\langle A,B \rangle$ be a discrete, free subgroup of $\ {\rm  SU }(3,1)$ that is generated by two loxodromic elements $A$ and $B$. Then  $\langle A,B \rangle$ preserves  a $\C^2$-plane if and only if  one of the following holds.
\medskip \begin{enumerate}
\item[(i)]{ $\zeta_o =0$ and,  either $\eta_1(A, B)=0=\nu_1(A, B)$ or $\eta_2(A, B)=0=\nu_2(A, B)$. }
\medskip \item[(ii)]{$\zeta_o=\infty$ and,  either $\eta_1(A, B)=0=\nu_2(A, B)$ or $\eta_2(A, B)=0=\nu_1(A, B)$. }
\end{enumerate}
\end{theorem}

\begin{proof}
Note that a two dimensional totally geodesic subspace of $\ch^3$ corresponds to a copy of $\C^{2,1}$. 

\medskip \textit{The condition is necessary}.   Suppose $\langle A,B \rangle$ preserve a copy of $\C^{2,1}$. Observe that $\langle A,B \rangle$ preserve a copy of $\C^{2,1}$ if and only if $A$ and $B$ have a common space-like eigenvector. Thus, either of the following cases arises:
\begin{itemize}
 \item[(a)] $x_A=x_B$;
 \item[(b)] $y_A=y_B$;
 \item[(c)] $y_A=x_B$;
 \item[(d)] $x_A=y_B$.
\end{itemize}
The result follows from the definition of $\eta_{i}(A,B)$'s and $\zeta_o(A,B)$.

\medskip \textit{The condition is sufficient}. 
Suppose $\zeta_o=0$. We discuss the case (i), i.e., let 
$$\eta_{1}(A,B)=0=\nu_{1}(A,B)=0=\zeta_o(A,B).$$ 
We claim that $x_A=x_B$.
We have,
\begin{equation*}
 \langle {\bf a}_A,{\bf x}_B \rangle \langle {\bf x}_B,{\bf r}_A \rangle =0;
\end{equation*}
\begin{equation*}
\langle {\bf a}_B,{\bf x}_A \rangle \langle {\bf x}_A,{\bf r}_B \rangle=0;
\end{equation*}
\begin{equation*}
\langle {\bf x}_B,{\bf y}_A \rangle \langle {\bf y}_B,{\bf x}_A \rangle =0.
\end{equation*}
Different subcases arise, it is enough to consider the following subcase:
 \begin{equation} \label{sc1} \langle {\bf a}_A,{\bf x}_B \rangle=0, \ \langle {\bf a}_B,{\bf x}_A \rangle=0, \   \langle {\bf y}_B,{\bf x}_A \rangle=0.\end{equation} 
Since,  $\{{\bf a}_B,~{\bf x}_{B},~{\bf y}_B,~{\bf r}_B\}$ is a basis for $\C^{3,1}$, hence,  there exists scalars $\mu_1,  \ \mu_2, \ \mu_3, \ \mu_4$ such that 
$${\bf x}_A=\mu_{1}{\bf a}_{B}+\mu_{2}{\bf x}_{B}+\mu_{3}{\bf y}_{B}+\mu_{4}{\bf r}_{B}.$$
The conditions $\langle {\bf a}_B,{\bf x}_A \rangle=0=\langle {\bf y}_B,{\bf x}_A \rangle$ 
implies $\mu_{3}=0=\mu_{4}$. Hence 
$${\bf x}_A=\mu_{1}{\bf a}_{B}+\mu_{2}{\bf x}_{B}.$$
This implies  
$$0=\langle {\bf x}_{A}, {\bf a}_A \rangle=\mu_{1}\langle {\bf a}_{B}, {\bf a}_{A} \rangle+\mu_{2}\langle {\bf x}_{B}, {\bf a}_{A} \rangle.$$
Using \eqnref{sc1} we have,   $\mu_{1}\langle {\bf a}_{B}, {\bf a}_{A} \rangle=0$. 
Since $\langle {\bf a}_{B}, {\bf a}_{A} \rangle\neq 0$, we have $\mu_{1}=0$. 
Hence ${\bf x}_A=\mu_{2}{\bf x}_{B}$ i.e., $x_B=x_A$, proving the result for the case (i). The argument in the other cases are similar.  

Note that if $\zeta_o=\infty$, then $1/\zeta_o=0$ and  similar arguments work in these cases also. 
\end{proof}
The subgroup $\langle A, B \rangle$ of ${\rm SU}(3,1)$ is called \emph{irreducible} or \emph{Zariski-dense} if it does not preserve a totally geodesic subspace of $\ch^3$. Using the above theorem and the results on cross-ratios, it is possible to derive many conditions for irreducibility of $\langle A, B \rangle$. As a special case we have the following. 
\begin{corollary}
 Let $A$ and $B$ be two loxodromic elements in ${\rm SU}(3,1)$ such that $\langle A, B \rangle$ is non-singular. Then $\langle A, B \rangle$ is irreducible. 
\end{corollary}
\section{Proof of \thmref{mainth}}\label{mnth}  
In this section, we follow the notations from \secref{li}. First, we shall show that for a non-singular pair $(A, B)$, one can always get a well-defined $\alpha$-invariant and a well-defined $\beta$-invariant. 

\subsection{$\alpha$ and $\beta$-invariants are well-defined} Let $A$ and $B$ be two loxodromics such that they form a non-singular pair. Without loss of generality, we can assume $A$ is a diagonal matrix, that is $C_A=\begin{bmatrix} {\bf e}_1 & { \bf e}_2 &  {\bf e}_3 & { \bf e}_4\end{bmatrix} $, where $\{ {\bf e}_1,~ {\bf e}_2, ~{\bf e}_3, ~{\bf e}_4 \}$ is the standard  basis of $\C^{3,1}$.  Let 
$B=C_B E(\lambda, \psi) C_B^{-1}$, where $C_B=\begin{bmatrix} \a_B &  \x_B &  \y_B &  \r_B\end{bmatrix}$. Let
$$\a_B=\begin{bmatrix} a \\ e \\ j \\ n \end{bmatrix}, ~ \x_B=\begin{bmatrix} b \\ f \\ k \\ s \end{bmatrix}, ~ \y_B=\begin{bmatrix} c \\ g \\ l \\ p \end{bmatrix}, ~ \r_B=\begin{bmatrix} d \\ h \\ m \\ q \end{bmatrix}. $$
Now we see that
$$\alpha_1(A, B)=\frac{nb}{as}, ~ ~ \alpha_2 (A, B)=\frac{nc}{ap},$$
$$ \beta_1(A, B)=\frac{\bar n \bar h}{\bar q \bar e}, ~ ~  \beta_2(A, B)=\frac{\bar n \bar m}{{\bar q}{ \bar j}}.$$
Since $\a_B$ and $\r_B$ are negative vectors, we must have $a$, $n$ and $q$ non-zeros. Now note that
\begin{equation} \label{eq1} \langle \a_A, \x_B \rangle =b, ~ \langle \r_A, \x_B \rangle=s,  \end{equation}
\begin{equation}\label{eq2} \langle \a_A, \y_B \rangle =c, ~\langle \r_A, \y_B \rangle =p, \end{equation} 
\begin{equation}\label{eq3} \langle \a_B, \x_A \rangle=e, ~ \langle \r_B, \x_A \rangle= h, \end{equation}
\begin{equation}\label{eq4} \langle \a_B, \y_A \rangle =j, ~ \langle \r_B, \y_A \rangle =m. \end{equation}

It follows from the condition (iii) in the definition of the non-singularity that neither of $\a_A$ and $\r_A$  belong to at least one of the $\C^2$-chains $\x_B^{\perp}$ and $\y_B^{\perp}$, and also, neither of $\a_B$ and $\r_B$ belong to one of the $\C^2$-chains $\x_A^{\perp}$ and $\y_A^{\perp}$. Thus,   one of the equations \eqnref{eq1} and \eqnref{eq2} must have entirely non-zero solution. Similarly, the solution to one of the equations \eqnref{eq3} and \eqnref{eq4} is also entirely non-zero.  Thus,  at least one $\alpha$-invariant and one $\beta$-invariant are always well-defined complex numbers for a non-singular pair of loxodromics. 

It can further be seen from the definition of Goldman's eta invariants that the well-definedness of $\alpha$-invariant and $\beta$-invariant can be stated equivalently by saying that for some $i, j \in \{1,2\}$,  $\eta_i(A, B) \neq 0$ and $\nu_j(A, B) \neq 0$.

\subsection{Proof of \thmref{mainth}}\;

\medskip 
\begin{lemma}\label{lsprp1}
 Let $A,~B,~ A',~B'$ be loxodromic elements in ${\rm {\rm SU }}(3,1)$. Let $\langle A, B \rangle$ be a non-singular subgroup in ${\rm SU}(3,1)$ such that for some $i, j \in\{1,2\}$,  $\eta_i(A, B) \neq 0$ and $\nu_j(A, B) \neq 0$.
Suppose $\alpha_i(A, B)=\alpha_i(A', B')$, $\beta_j(A, B)=\beta_j(A', B')$ and,  for $k=1,2,3$,   $\X_{k}(A,~B)=\X_{k}(A',~B')$.  Then
there exists an element $C$ in ${\rm SU }(3,1)$ such that $C(a_{A})=a_{A'},C(x_{A})=x_{A'},~C(y_{A})=y_{A'},~C(r_{A})=r_{A'}$,  and, 
$C(a_{B})=a_{B'},C(x_{B})=x_{B'},~C(y_{B})=y_{B'},~C(r_{B})=r_{B'}$.  
\end{lemma}
\begin{proof}
We shall prove the lemma assuming that $(i, j)=(1,1)$.  The rest of the cases are similar.  Since $~\X_{k}(A,~B)=\X_{k}(A',~B'),~k=1,2,3$, by \propref{crvpr3} it follows that there exist $C\in{\rm SU }(3,1)$ such that 
$a_{A'}=C(a_{A}),~r_{A'}=C(r_{A}),~a_{B'}=C(a_B)$ and  $r_{B'}=C(r_{B})$. 
Since $\alpha_1(A',~B')=\alpha_1(A,~B)$, we have
\begin{eqnarray*}
\frac{\langle \xb, \ra \rangle \langle \ab, \aa\rangle}{\langle \ab, \ra \rangle \langle \xb, \aa \rangle }& = & \frac{\langle \xb', \ra'\rangle \langle \ab', \aa' \rangle}{\langle \ab', \ra' \rangle \langle \xb', \aa'\rangle}\\
&=& \frac {\langle C^{-1} (\xb'), \ra\rangle \langle \ab, \aa\rangle}{\langle \ab, \ra \rangle \langle C^{-1}(\xb'), \aa \rangle}\\
\implies \frac{\langle \xb, \ra \rangle}{\langle  C^{-1} (\xb'), \ra\rangle}
 &=& \frac{\langle \xb, \aa \rangle}{\langle C^{-1}(\xb'), \aa \rangle}.
\end{eqnarray*}
Let 
$$\lambda=\frac{\langle \xb, \ra \rangle}{\langle  C^{-1} (\xb'), \ra\rangle}
 = \frac{\langle \xb, \aa \rangle}{\langle C^{-1}(\xb'), \aa \rangle}.$$
This implies
\begin{equation}\label{e1}  \langle \xb-\lambda C^{-1}(\xb'), \ra \rangle=0; \end{equation} 
\begin{equation} \label{e2}  \langle \xb-\lambda C^{-1}(\xb'), \aa \rangle=0.\end{equation}
On the other hand, note that
\begin{equation}\label{e3}  \langle \xb-\lambda C^{-1} (\xb'), \rb\rangle=\langle \xb, \rb\rangle-\overline \lambda\langle C^{-1}(\xb')-\rb \rangle=0-\overline \lambda \langle \xb', \rb' \rangle=0. \end{equation} 

Similarly,
\begin{equation} \label{e4} \langle  \xb-\lambda C^{-1} (\xb'), \ab \rangle=0.\end{equation}
Let $L_A$  and $L_B$ denote the two-dimensional time-like subspaces of $\C^{3,1}$ that represent the complex axes of $A$ and $B$ respectively. Thus $\{\aa, \ra\}$ and $\{\ab, \rb\}$ are the respective bases of $L_A$ and $L_B$. It follows from \eqnref{e1} -- \eqnref{e4} that $v= \xb-\lambda C^{-1} (\xb')$ is orthogonal to  both $L_A$ and $L_B$.  We must have $\langle v, v \rangle >0$. Thus $v$ is polar to the $\C^2$-chain (copy of $\ch^2$)  that is represented by ${\rm V}=v^{\perp}$. Since $\C^{3,1}={\rm V} \oplus  \C v $, hence $L_A$ and $L_B$ must be subsets in ${\rm V}$. Thus, the fixed points of $A$ and $B$ belong to the boundary of the $\C^2$-chain $\P({\rm V})$. This is a contradiction to the non-singularity of $(A, B)$. Hence we must have $v=0$, that is $C(\xb)=\lambda \xb'$. Thus,  $C(x_B)=x_B'$. Consequently, $C(y_B)=y_B'$. 

Similarly $\beta_1(A, B)=\beta_1(A', B')$ implies $C(x_A)=x_A'$, and hence, $C(y_A)=y_A'$.  This proves the lemma. 
\end{proof}

\subsubsection{Proof of \thmref{mainth}}
\begin{proof}
Suppose that $A,~B,~A',~B'$  are loxodromic elements such that  $$tr(A)=tr(A'),~tr(B)=tr(B'),~\sigma(A)=\sigma(A'), ~  \sigma(B)=\sigma(B');$$
$$\alpha_i(A, B)=\alpha_i(A', B'), ~\beta_j(A, B)=\beta_j(A', B') ~ \hbox{ and for }k=1,2,3, ~ \X_{k}(A,~B)=\X_{k}(A',~B'). $$
 Following the notation in \secref{li}, $A=C_{A}E_{A}C_{A}^{-1},~B=C_{B}E_{B}C_{B}^{-1}$ and similarly for $A'$
 and $B'$. Since the cross-ratios are equal, by \lemref{lsprp1} it follows that there exist $C\in{\rm{{\rm SU }}}(3,1)$ such that 
$C(a_{A})=a_{A'},C(x_{A})=x_{A'},~C(y_{A})=y_{A'},~C(r_{A})=r_{A'}$ and $C(a_{B})=a_{B'},C(x_{B})=x_{B'},~C(y_{B})=y_{B'},~C(r_{B})=r_{B'}$. Therefore $CAC^{-1}$
 and $A'$ have same eigenvectors. Since $tr(A')=tr(CAC^{-1}),~\sigma(A')=\sigma(CAC^{-1})$, by \corref{licor1} and \propref{lipr1}, we must have $CAC^{-1}=A'$.
 Similarly, \hbox{$B'=CBC^{-1}$}.  Thus $\langle A',B' \rangle=\langle CAC^{-1},CBC^{-1} \rangle=C\langle A,B \rangle C^{-1}$ as claimed.   
\end{proof}

\section{The Twist-Bend Parameter} \label{twb} 
 Let $\langle A, B \rangle$ be a non-singular $(0,3)$ group in ${\rm SU}(3,1)$, that is, $A$ and $B$ are loxodromics such that $AB$ is also loxodromic. We want to attach two such non-singular subgroups to get a group that is freely generated by three generators. Now two cases are possible.  The first case corresponds to the case when two different pairs of pants are attached along their boundary components. In this case, the $(0,3)$ groups correspond to different pairs of pants and, they give a $(0, 4)$ group. The second case corresponds to the case when two of the boundary components of the same pair of pants is attached to give a torus. In this case attaching two $(0,3)$ groups yields an $(1,1)$ group which is a group generated by two loxodromic elements and their commutator. This process is called `closing a handle'. To get more details about the geometric description of these terminologies and their interpretations in terms of group theoretic operations, we refer to Parker-Platis \cite{pp}. 

Unless stated otherwise, the two-generator subgroups in this section are always assumed to be non-singular. 
Let $\langle A, B \rangle$ and $\langle C, D \rangle$ be two such $(0,3)$ groups in ${\rm SU}(3,1)$ such that the boundary components associated to $A$ and $D$ are compatible, i.e., $A=D^{-1}$. A \emph{complex hyperbolic twist bend} corresponds to an element $K$ in ${\rm SU}(3,1)$ that commutes with $A$ and conjugates $\langle C, D \rangle$, see Parker-Platis \cite[Section 8.1]{pp} for the ideas behind this notion. We assume that up to conjugacy, $A$ fixes $0$, $\infty$, and it is of the form $E(\lambda, \phi)$, for some $(\lambda, \phi) \in S$. Since $K$ commutes with $A$, it is also of the form $K=E(\kappa, \psi)$ for some $(\kappa, \psi) \in S$, see \cite{gcpr}. Thus $K$ is either a boundary elliptic or, a loxodromic. The parameters  $(\kappa, \psi)$ obtained this way, is the \emph{twist-bend parameter}. It should be noted that the twist-bend is a relative invariant. It should always be chosen with respect to some fixed group $\langle A, B, C \rangle$ that one has to specify before applying the twist-bend. It 
gives us the degrees of freedom that are needed while attaching boundaries of pairs of pants to obtain a two-holed sphere that is fixed at the beginning. When we write $A=Q E(\lambda, \phi) Q^{-1}$, if the matrix $K$ is given by $QE(\kappa, \psi)Q^{-1}$, then we say that the twist-bend parameter $(\kappa, \psi)$ is \emph{oriented consistently} with $A$.  

To obtain conjugacy-invariant way to measure the twist-bend parameter, we define the following quantities:
$$\tilde \X_1(\kappa, \psi)=[\a_B, \a_A, \r_A, K(\r_C)], ~ \tilde \X_2(\kappa, \psi)=[\a_B, \r_A, \a_A, K(\r_C)], 
$$
$$\tilde \beta_1(\kappa, \psi)=[K(\r_C), \a_B, \x_A, \a_A], ~ \tilde \beta_2(\kappa, \psi)=[K(\r_C), \a_B, \y_A, \a_A].$$
\begin{lemma}\label{tw1} 
 Let $A$, $B$, $C$ be loxodromic elements in ${\rm SU}(3,1)$ such that $\langle A, B \rangle$ and $\langle A^{-1}, C \rangle$ are non-singular.  Let $(\kappa, \psi)$ and $(\kappa', \psi')$ are twist-bend parameters that are oriented consistently with $A$. If 
 $$\tilde \X_1(\kappa, \psi)=\tilde \X_1(\kappa', \psi'), ~ \tilde \X_2(\kappa, \psi)=\tilde \X_2(\kappa', \psi'), \hbox{ and  } ~  \tilde \beta_i(\kappa, \psi)=\tilde \beta_i(\kappa', \psi'), \hbox{ i =1 or 2},$$
 then $\kappa=\kappa',  \psi=\psi'$. 
\end{lemma}
\begin{proof}
We shall prove the lemma assuming  $\nu_1(A, C) \neq 0$. The other case is similar. 

 Without loss of generality, assume that $A$ fixes $0$ and $\infty$ and up to conjugacy $A=E(\lambda, \phi)$. So, we can further assume $K=E(\kappa, \psi)$. 
 Let $a_A=\infty$, $r_A=0$. Thus 
 $$\a_A=
 \begin{bmatrix} 1 \\ 0\\ 0\\ 0 \end{bmatrix}, ~ \r_A=\begin{bmatrix} 0 \\ 0\\ 0\\ 1\end{bmatrix}.$$
 Let 
 $$\a_B=\begin{bmatrix} a \\ d\\g\\h \end{bmatrix}, ~ \r_B=\begin{bmatrix} c\\ f\\j\\t \end{bmatrix}, ~ \r_C=\begin{bmatrix} c' \\ f'\\j'\\t' \end{bmatrix}.$$
Further we assume without loss of generality, $\x_A=\begin{bmatrix} 0\\ 1\\0\\0 \end{bmatrix}$. Since $\a_A$, $\r_B$, $\r_C$ are light-like vectors,  $a, h, t', c'$ are non-zero numbers. Hence,  we have,
\begin{eqnarray}
 \frac{\tilde \X_1(\kappa, \psi)}{\tilde \X_2(\kappa, \psi)}&=&\frac{\tilde \X_1(\kappa', \psi')}{\tilde \X_2(\kappa', \psi')}\\
\Rightarrow \frac{t' e^{-\bar \kappa} \bar a}{ c' e^{\kappa} \bar h} &=& \frac{t' e^{-\bar \kappa'} \bar a}{c' e^{\kappa} \bar h}\\
\Rightarrow \kappa=\kappa'.
\end{eqnarray}
Next $\tilde \beta_1(\kappa, \psi)=\tilde \beta_1(\kappa', \psi')$ implies
$$\frac{\bar f' e^{i \psi + \frac{\kappa-\bar \kappa}{2}}}{\bar t' e^{-\kappa}}=\frac{\bar f' e^{i \psi' + \frac{\kappa'-\bar \kappa'}{2}}}{\bar t' e^{-\kappa'}}.$$
As $\nu_1(A, C) \neq 0$, $f' \neq 0$ and since $\kappa=\kappa'$, we must have $\psi=\psi'$.   
\end{proof}

\subsection{Attaching two pairs of pants} A $(0, 4)$ subgroup of ${\rm SU}(3,1)$ is a group with four loxodromic generators such that their product is identity. These four loxodromic maps correspond to the boundary curves of the four-holed spheres and are called \emph{peripheral}.  Thus a $(0,4)$ group is freely generated by any of these three loxodromic elements.

Let $\langle A, B \rangle$ and $\langle C, D \rangle$ are two $(0,3)$ groups with $A^{-1}=D$. Algebraically, a $(0,4)$ group is constructed by the amalgamated free product of these groups with amalgamation along the common cyclic subgroup $\langle A \rangle$. Conjugating $\langle C, D \rangle$ by the twist-bend $K$ yields a new $(0, 4)$ subgroup that is depended on $K$.  We note the following lemma whose proof goes the same as Lemma 8.3 of Parker-Platis \cite[p.131]{pp}. 

\begin{lemma}
Suppose $\Gamma_1=\langle A, B \rangle$ and $\Gamma_2=\langle C, D \rangle$ are two $(0,3)$ groups with peripheral elements $A$, $B$, $B^{-1} A^{-1}$ and $C$, $D$, $D^{-1} C^{-1}$ respectively. Moreover suppose that $A=D^{-1}$. Let $K$ be any element of ${\rm SU}(3,1)$ that commutes with $A=D^{-1}$. The the group $\langle A, B, KCK^{-1} \rangle$ is a $(0, 4)$ group with peripheral elements $B$, $B^{-1} A^{-1}$, $KCK^{-1}$ and $KD^{-1} C^{-1} K^{-1}$. 
\end{lemma}

\begin{prop}\label{tb1} 
Suppose that $\langle A, B \rangle$ and $\langle C, A^{-1} \rangle$ are two non-singular $(0,3)$ groups. Let $(\kappa, \psi)$ be a twist-bend parameter oriented consistently with $A$ and let $\langle A, B, KCK^{-1}\rangle$ be the corresponding $(0, 4)$ group. Then $\langle A, B, KCK^{-1} \rangle$ is uniquely determined up to conjugation in ${\rm SU}(3,1)$ by the \emph{Fenchel-Nielsen coordinates}: 

\medskip 

 $tr(A),~tr(B),~tr(C), ~\sigma(A),~\sigma(B),~\sigma(C)$, 
$~\X_{k}(A,~B), ~\X_k(A, C), ~k=1,2,3$, two non-zero $\alpha$-invariants: $\alpha_i(A, B), ~ \alpha_l(A, C)$, two non-zero $\beta$-invariants: $\beta_j(A, B), ~ \beta_m(A, C)$ and the twist-bend parameter $(\kappa, \psi)$. 

\medskip In the parameter space associated to $\langle A, B , KCK^{-1}\rangle$, the parameters corresponding to $traces, ~ \alpha$ and $\beta$ are complex numbers, the tuple $(\X_1, \X_2, \X_3)$ belongs to the cross-ratio variety that is 5 real dimensional and $(\kappa, \psi)$ has real dimension three. Thus we need a total of 30 real parameters to specify $\langle A, B, KCK^{-1}\rangle$ up to conjugacy. 
\end{prop} 

The number 30 is obtained by the following count: 

$[(3~(\hbox{traces}) \times 2  + (3~ (\hbox{$\sigma$-invariants}) \times 1 )+ (2 ~(\hbox{cross-ratios}) \times 5)+ (2 ~(\hbox{$\alpha$-invariants}) \times 2)+ (2 ~(\hbox{$\beta$-invariants}) \times 2) + ( 1~(\hbox{twist-bend}) \times 3)=30]$.

\begin{proof}
Suppose, $\langle A, B, KCK^{-1} \rangle$ and $\langle A', B', K'C'K'^{-1}\rangle$ are two $(0,4)$ subgroups having the same Fenchel-Nielsen coordinates. Let the following relations hold: 
\begin{equation}\label{t1} tr(A)=tr(A'), ~ tr(B)=tr(B'),~ tr(C)=tr(C'); \end{equation} 
\begin{equation} \label{t2} \sigma(A)=\sigma(A'), ~ \sigma(B)=\sigma(B'), ~ \sigma(C)=\sigma(C'); \end{equation} 
\begin{equation}\label{t3} \hbox{for }t=1,2,3, ~ \X_t(A, B)=\X_t(A', B'), ~ \X_t(A, C)= \X_t(A', C');\end{equation} 
\begin{equation} \label{t41} \alpha_i(A, B)=\alpha_i(A', B'), ~ ~ \alpha_l(A, C)=\alpha_l(A', C'); \end{equation} 
\begin{equation} \label{t42}  \beta_j(A, B)=\beta_j(A', B'), ~ \beta_m(A, C)=\beta_m(A', C');\end{equation} 
\begin{equation} \label{t5} (\kappa, \psi)=(\kappa', \psi').\end{equation} 
Using these relations, it follows from \thmref{mainth} that there exist $C_1$ and $C_2$ in ${\rm SU}(3,1)$ that conjugate $\langle A, B \rangle$ and $\langle C, A^{-1} \rangle$ respectively to $\langle A', B'\rangle$ and $\langle C', A'^{-1} \rangle$.

  Now the twist-bends are defined with respect to the same initial group $\langle A, B, C \rangle$ that we fix at the beginning before  attaching the two $(0,3)$ groups. So,  without loss of generality, we may assume that $A=A', ~ B=B', ~ C=C'$, and thus,  $C_1=C_2$. Now with respect to the same initial  group $\langle A, B, C \rangle$, by \eqnref{t5} it follows that $K=K'$. This implies that $\langle A, B, KCK^{-1} \rangle$ is determined uniquely up to conjugacy. 

Conversely, suppose that $\langle A, B, KCK^{-1} \rangle$ and $\langle A', B', K'C'K'^{-1} \rangle$ are conjugate. Then clearly, \hbox{\eqnref{t1} --\eqnref{t42} } are satisfied. The only thing remains to show is \eqnref{t5}. Now by the invariance of the cross-ratios it follows that
$$\tilde \X_1(\kappa, \psi)=\tilde \X_1(\kappa', \psi'), ~\tilde \X_2(\kappa, \psi)=\tilde \X_2(\kappa', \psi'), ~ \tilde \beta_k(\kappa, \psi)=\tilde \beta_k(\kappa', \psi'),$$
and hence by \lemref{tw1}, $(\kappa, \psi)=(\kappa', \psi')$. 
\end{proof}
\subsection{Closing a handle} We are now interested in obtaining a one-holed torus by attaching two holes of the same pair of pants in the complex hyperbolic $3$-space. The process of attaching these two holes is called \emph{closing a handle}. Geometrically, it corresponds to attach two boundary components of the same pair of pants. To make this work, one of the peripheral elements of the corresponding $(0,3)$ group must be conjugate to the inverse of the other peripheral element. This ensures that they are compatible for the attachment. Suppose the two peripheral elements are $A$ and $BA^{-1} B^{-1}$, then the third element would be $[B, A]=BAB^{-1} A^{-1}$. A $(1,1)$ subgroup of ${\rm SU}(3,1)$ is a group that is generated by the elements $A, ~ B$ and $[A, B]$. From a group theoretic viewpoint, closing a handle is the same as taking the HNN-extension of the $(0, 3)$ group $\langle A, BA^{-1} B^{-1} \rangle$ by adjoining the element $B$ to form a $(1,1)$ group. When we take the HNN-extension, the map $B$ is not 
unique. If $K$ is any element in ${\rm SU}(3,1)$ that commutes with $A$, then $\langle A, BK \rangle$ gives another $(1,1)$ group. Varying $K$ corresponds to a twist-bend coordinate as above. 

 If $A=QE(\lambda, \phi) Q^{-1}$ for $(\lambda, \phi) \in S$, just as before, we define the twist-bend parameter  $(\kappa, \psi)$ by 
$K=QE(\kappa, \psi) Q^{-1}$, and we say, $(\kappa, \psi)$ is oriented consistently with $A$. In this case also $(\kappa, \psi)$ is defined relative to a reference group that we fix at the starting of the attachment. 

\begin{lemma}\label{ch}
Let $\langle A, BA^{-1} B^{-1} \rangle$ be a non-singular $(0, 3)$ group. Let $B$ be a fixed choice of an element in ${\rm SU}(3,1)$ conjugating $A^{-1}$ to $BA^{-1} B^{-1}$. Let $(\kappa, \psi)$ $(\kappa', \psi')$ be twist-bend parameters oriented consistently with $A$. Then $\langle A, BK\rangle$ is conjugate to $\langle A, BK'\rangle$ if and only if $(\kappa, \psi)=(\kappa, \psi')$. 
\end{lemma}
\begin{proof} 
If $(\kappa, \psi)=(\kappa, \psi')$, then clearly $K=K'$ and hence the groups are equal. 

Conversely, suppose $\langle A, BK \rangle$ is conjugate to $\langle A, BK'\rangle$. The conjugating element $D$ must commutes with $A$. Hence $D(a_A)=a_A, ~ D(r_A)=r_A$.  Since $BA^{-1} B^{-1}$ has been fixed at the beginning, we have
$$BA^{-1} B^{-1}=(BK') A^{-1} (BK')^{-1}=(DBKD^{-1})A^{-1} (DBKD^{-1})^{-1}=D(BA^{-1} B^{-1})D^{-1}.$$
Thus $D$ commutes with $BA^{-1} B^{-1}$ and fixes $a_{BA^{-1} B^{-1}}=B(r_A), ~ r_{BA^{-1} B^{-1}}=B(a_A)$. Since the fixed points are distinct, $D$ is either the identity or,  the fixed points $a_A, ~ r_A, ~B(a_A),~B(r_A)$ belong to the same chain fixed by $D$. But the later is not possible by the non-singularity of the $(0,3)$ group. Thus,  $D$ must be the identity. Thus, $BK'=BK$ and hence, $K=K'$, i.e., $(\kappa, \psi)=(\kappa', \psi')$. 
\end{proof} 
\begin{prop}\label{tw2} 
Let  $\langle A, BK \rangle$ be a $(1,1)$ group obtained from the non-singular $(0,3)$ group $\langle A, BA^{-1} B^{-1} \rangle$ by closing a handle with associated twist-bend parameter $(\kappa, \psi)$. Then $\langle A, BK \rangle$ is determined uniquely up to conjugation by its \emph{Fenchel-Nielsen coordinates} 

\medskip 

 $tr(A),~\sigma(A)$, 
$~\X_{k}(A, BA^{-1} B^{-1}),~k=1,2,3$, one non-zero  $\alpha$-invariant: $\alpha_i(A, BA^{-1} B^{-1})$, one non-zero $\beta$-invariant: $\beta_j(A, BA^{-1} B^{-1})$  and the twist-bend parameter $(\kappa, \psi)$. 

\medskip Thus,  we need 15 real parameters to specify $\langle A, BK \rangle$ up to conjugacy. 
\end{prop}
\begin{proof}
Suppose that $\langle A, BK \rangle$ and $\langle A, B'K' \rangle$ are two $(1,1)$ groups with the same Fenchel-Nielsen coordinates. In particular $tr(A)=tr(A')$ and hence
$$tr(BA^{-1} B^{-1})=\overline{ tr(A)}=\overline{tr(A')}=tr(B'A'^{-1} B^{-1}).$$
Further using 
$X_k(A, BA^{-1} B^{-1})=\X_k(A', B'A'^{-1} B'^{-1}) \hbox{ for } k=1,2,3,$ 

\noindent $\alpha_i(A, BA^{-1} B^{-1})=\alpha_i(A', B'A'^{-1} B'^{-1}) \hbox{ and } \beta_j (A, BA^{-1} B^{-1})=\beta_j(A', B'A'^{-1} B'^{-1})$, 
we see by \thmref{mainth} that the $(0, 3)$ groups $\langle A, BA^{-1} B^{-1} \rangle$ and $\langle A', B'A'^{-1} B'^{-1} \rangle$ are conjugate. Thus we can assume $A=A', ~  BA^{-1} B^{-1}=B'A'^{-1} B'^{-1}$. Now using the above lemma,  we see that $(\kappa, \psi)=(\kappa', \psi')$. Hence $K=K'$. Thus, the group $\langle A, BK\rangle$ is determined uniquely up to conjugation. 

Conversely, suppose $\langle A, BK \rangle$ and $\langle A', B'K' \rangle$ are conjugate. Hence, it is clear that all the Fenchel-Nielsen coordinates but the twist-bend parameters are the same. Conjugating, if necessary, we assume $A=A'$.  Since $B$ is a fixed choice of the conjugation element with reference to which $(\kappa, \psi)$ and $(\kappa', \psi')$ are defined, we may also assume $B=B'$. Now, using \lemref{ch},  we see that $K=K'$, i.e., $(\kappa, \psi)=(\kappa', \psi')$. 
\end{proof}
\section{Proof of \thmref{mth2}} \label{mnth2} 
We have seen in \thmref{mainth} that each $(0, 3)$ group $\langle A, B \rangle$ is determined up to conjugacy by $15$ real parameters. Each pair of pants in the pants-decomposition (complement of $\mathcal C$) corresponds to a $(0, 3)$ group. While we attach two pairs of pants, we attach two $(0, 3)$ groups subject to the compatibility condition that one peripheral element in a group is conjugate to the inverse of a peripheral element in the other group. Thus,  we get a $(0, 4)$ group that is specified up to conjugacy by the $30$ real parameters described in  \propref{tb1}. Continuing this way, when we attach  $2g-2$ of our $(0, 3)$ groups, we need a total $15(2g-2)=30g-30$ real parameters to specify the resulting surface with $2g$ handles. These handles correspond to the $g$ curves that correspond to the two boundary components of the same three-holed sphere. Now,  there are $g$ complex constraints which are  imposed to close these handles: one of the peripheral elements of each of these $(0,3)$ groups  must be conjugate to the inverse of the other peripheral element. Note that,  to each peripheral element there are $3$ natural real parameters: the trace and the $\sigma$-invariant. So, the number of real parameters now reduced to $30g-30-3g=27g-30$. But there are $g$ twist-bend parameters $(\kappa_i, \psi_i)$, one for each handle,  and each contributes $3$ real parameters. Thus, we need a total of $27g-30 + 3g=30g-30$ real parameters to specify $\rho$ up to conjugacy.

If two representations have the same coordinates, then the coordinates of the $(0,3)$ groups are the same, so they are conjugate. Further it follows from \propref{tb1} and \propref{tw2} that the $(0,4)$ groups and the $(1,1)$ groups are also determined uniquely up to conjugacy while attaching the $(0,3)$ groups. Hence, the representations having the same parameters are conjugate. Conversely, if two representations are conjugate, then clearly they have the same coordinates. 

This proves the theorem.

\subsection{Remarks}
In a recent work, Gongopadhyay and Lawton \cite{gl} have given a collection of 22 trace coordinates subject to 5 real relations that determines any representation $\rho$ in ${\rm  Hom}({\rm F}_2, {\rm SU}(3,1))/{\rm SU}(3,1)$ having a closed conjugation orbit.  In particular, this coordinate system determines any Zariski-dense pair in ${\rm SU}(3,1)$. This gives 39 real coordinates to determine such representations.  At the same time, it has been shown that the real dimension of the \hbox{smallest}  possible system of such real parameters is 30.  It would be interesting to reduce the 39 dimensional coordinate system to a 30 dimensional system. To achieve this, it is required to obtain relations involving the above trace coordinates. Currently, except the above mentioned five real relatons, no other relation is known.

In view of the above,  \thmref{mainth} of this paper provides a smaller number of coordinates that is enough to determine generic pairs of loxodromic elements in ${\rm SU}(3,1)$.  By the results in \cite{gl}, it is clear that a 15 real dimensional coordinate system can not determine pairs in ${\rm SU}(3,1)$. However, it is interesting to obtain classes of pairs  that may be determined by lesser dimensional parameters.  It would also be interesting to generalize this work for $n>3$. In higher dimensions, the approach in this paper may be difficult to generalize because of the obstruction to obtain sufficient number of conjugacy invariants. A starting point is the classification of pairs of elements in ${\rm SU}(n, 1)$ up to conjugacy. Following classical invariant theory, one approach is to obtain this classification using trace invariants, eg. the coefficients of the characteristic polynomials. This problem is of considerable difficluty even in the case of ${\rm SL}(n, \C)$, where  a complete set of minimal number of trace parameters along with relations between them,  is still unknown except for a few lower values of $n$, see for eg. \cite{ds, do}. 

\def\cdprime{$''$} \def\Dbar{\leavevmode\lower.6ex\hbox to 0pt{\hskip-.23ex
  \accent"16\hss}D}

\end{document}